\documentclass[preprint]{elsarticle}
\biboptions{longnamesfirst,square,semicolon}
\title{On the determinant of hexagonal grids $H_{k,n}$}
\author[AB]{Anna Bie\'n}
\ead{anna.bien@us.edu.pl}
\address{Institute of Mathematics, University of Silesia, Katowice, Poland}

\usepackage[T1]{fontenc} 
\usepackage[bookmarksnumbered,plainpages]{hyperref} 
\usepackage{amsfonts} 
\usepackage[cp1250]{inputenc} 
\usepackage{amsthm}

\newtheorem{tw}{Theorem}

\newtheorem{co}{Corollary}

\begin{document}

\begin{abstract}
We analyse the problem of singularity of graphs for hexagonal grid graphs.
We introduce methods for transforming weighted graph, which do not change the determinant of adjacency matrix. 
We use these methods to calculate the determinant of all hexagonal grid graphs which describe hexagon-shaped benzenoid systems $O(n,1,k)$. The final result is the explicit formula for the determinant of graphs $H_{k,n}$. From the theorem we draw the conclusion, that all graphs of this kind are non-singular.

\end{abstract}

	\maketitle

\section{Introduction}

\emph{Weighted graph} $G$ is a triplet $(V(G), E(G), w^G)$, such that $(V(G),E(G))$ is a simple graph, and $w^G$ is a weight function, which associates certain value with every edge. It is convenient for us to define the weight as a function on the set $E(G) \cup E(\overline{G})$ in such way that $w(e)\neq 0$ iff $e \in E(G)$. $\overline{G}$ denotes the complement graph of $(V(G),E(G))$. 
\sloppy
We identify a simple graph $(V(G),E(G))$ with the weighted graph $(V(G),E(G),w^G)$, where $w^G(e)=1$ iff $e \in E(G)$. We will use the term graph to describe a weighted graph.

If $|V(G)|=n$, then the adjacency matrix \mbox{$A(G)=[a_{i,j}]_{n \times n}$} is defined in the following way:
\fussy
	\begin{displaymath}
	a_{i,j} = \left\{ 
	\begin{array}{ll}
w(v_iv_j) & \textrm{\qquad $i \neq j$} \\
0 & \textrm{\qquad $i=j$}
\end{array} \right.
\end{displaymath} 
If $\hbox{det} A(G)=0$, then the graph $G$ is \emph{singular}.
A graph $H$ is a subgraph of a graph $G$ if $V(H) \subset V(G)$, $E(H) \subset E(G)$, and $w^G (e) = w^H (e)$ for every $e \in E(H)$.
$H$ is an induced subgraph if the underlying simple graph $(V(H),E(H))$ is an induced subgraph of the simple graph $(V(G),E(G))$. 
We similarly use other terms and notation of Diestel \cite{D}.

By removing an edge $e$ from a graph $G$ we mean removing the edge from the set of edges $E(G)$, which is equivalent to changing the weight of this edge to $0$. By removing a vertex $v$ from a graph $G$ we mean removing the vertex from the set of vertices, removing all edges incident to $v$, and restricting the weight function to the set $E(G-v) \cup E(\overline{G-v})$. $G-H$ denotes a graph obtained from the graph $G$ by removing all vertices of the graph $H$.

\section{Preliminaries}

A graph $G$ is \emph{sesquivalent} if all of its components are $1$- or $2$-regular graphs, 
i.e. every component is either a cycle or a path $P_2$. A sesquivalent graph is also called a basic figure.\cite{S} A graph $H$ which is a spanning subgraph of $G$ and a basic figure is called a sesquivalent spanning subgraph.  In literature sesquivalent spanning subgraphs are also called perfect 2-matchings, or Sachs graphs (in chemistry)\cite{G}. $S(G)$ denotes the set of all sesquivalent spanning subgraphs of the graph $G$.

In \cite{H} Harary used the variable determinant of a digraph to prove, that the determinant of the adjacency matrix of a simple graph can be calculated by applying the formula 
\begin{displaymath}
\hbox{det} A(G) =  \sum_{\Gamma \in S(G)} (-1)^{sg(\Gamma)}2^{c(\Gamma)}
\end{displaymath}
where $c(\Gamma)$ is the number of components of $\Gamma$ which are cycles, and $sg(\Gamma)$ is the number of components of $\Gamma$ with even number of vertices. We will use the following notation: $P^2(\Gamma)$ denotes the set of edges of $1$-regular components of $\Gamma$, and $C(\Gamma)$ denotes the set of edges of $2$-regular components of $\Gamma$.

The variable determinant can be applied to calculate the determinant of any weighted graph.
\begin{co}\label{w}
If $G=(V,E,w)$ is a weighted graph, then
\begin{displaymath}
\hbox{det} A(G)= \sum_{\Gamma \in S(G)} \left[ (-1)^{sg(\Gamma)} 2^{c(\Gamma)}\prod_{e \in P^2(\Gamma)} w(e)^2 \prod_{e \in C(\Gamma)} w(e) \right]
\end{displaymath}
\end{co}
The above formula was presented in \cite{S} after the observation made by Cvetković \cite{C} that it can be interpreted as an intuitive form of the Leibniz definition of the determinant.
The obvious consequence of this corollary is that we can remove all edges from a graph $G$ which do not lie in any sesquivalent spanning subgraph of $G$, because this operation does not influence the determinant of the graph. We immediately obtain the following conclusion.

\begin{co}
If $v \in V(G)$, $H$ is the only connected sesquivalent subgraph of $G$, such that $v \in V(H)$, then
$\hbox{det} A(G)= \hbox{det} A(H) \cdot \hbox{det} A(G-H)$.
\end{co}

\begin{proof}
$\Gamma \in S(G)$ iff $\Gamma=H \cup \Gamma'$, for some $\Gamma' \in S(G-H)$.
\end{proof}

If a graph $G$ has a pedant vertex, then a $P_2$ graph is the only connected sesquivalent subgraph of $G$ to which that pedant vertex belongs. Hence 

\begin{co}\label{ped}
If $v$ is a pedant vertex in a graph $G$, and $u$ is the vertex adjacent to $v$, then
$$\hbox{det} A(G) = \hbox{det} A(G \setminus \{ux: x \in N(u) \wedge u\neq v \})
= -\left[ w^G(vu)\right]^2 \cdot \hbox{det} A(G \setminus \{v,u\})$$
\end{co}

As a consequence of the corollary \ref{w} we also obtain formulas for the determinants of weighted paths and cycles.

\begin{co}
If $w$ is a weight function of $P_n=v_1v_2 \dots v_n$, then
\begin{displaymath}
\hbox{det} A(P_n)=\left\{ \begin{array}{ll}
0 & \textrm{if $2 \not| n$} \\
\displaystyle (-1)^{n/2}\prod_{e \in E_1(P_n)} w(e)^2  & \textrm{if $2|n$} 
\end{array} \right.
\end{displaymath}
where $E_1(P_n)=\{v_{i-1}v_{i} \in E(P_n): 2|i \ \wedge \ i \leq n  \}$.
\end{co}

\begin{proof}
If $n$ is odd, then $S(P_n)=\emptyset$. If $n$ is even, then $S(P_n)=\{\Gamma\}$, where all components of $\Gamma$ are paths $P_2$, and $E(\Gamma)=E_1(P_n)$.
\end{proof}

\begin{co}
If $w$ is a weight function of $C_n=v_1v_2 \dots v_nv_1$, then
\begin{displaymath}
\hbox{det} A(C_n)=\left\{ \begin{array}{ll}
\displaystyle 2 \prod_{e \in E(C_n)} w(e) & \textrm{if $2 \not| n$} \\
\displaystyle - \left[ \prod_{e \in E_1(C_n)} w(e) \ + \prod_{e\in E_2(C_n)}  w(e) \right]^2 & \textrm{if $2|n \wedge 4 \not| n$} \\
\displaystyle \ \left[\prod_{e \in E_1(C_n)} w(e) \ - \prod_{e\in E_2(C_n)}  w(e)\right]^2 & \textrm{if $4|n$}
\end{array} \right.
\end{displaymath}
where $E_1(C_n)=\{v_{i-1}v_{i} \in E(C_n): 2|i \ \wedge \ i \leq n  \}$ and $E_2(C_n)=E(C_n) \setminus E_1(C_n)$.
\newline
\end{co}

\begin{proof}
If $2\not| n$, then $S(C_n)=\{C_n\}$, and  
\begin{displaymath} 
\hbox{det}A(C_n)=  2 \prod_{e \in E(C_n)} w(e) 
\end{displaymath}  
If $2|n$, then $S(G)=\{C_n, \Gamma_1, \Gamma_2 \}$, where $E(\Gamma_1)=E_1(C_n)$ and $E(\Gamma_2)=E_2(C_n)$.
Hence, 
\begin{displaymath} 
\hbox{det}A(C_n)= \left[ -2  \prod_{e \in E(C_n)} w(e) \right]  + 
\left[ (-1)^{sg(\Gamma_1)}  \prod_{e \in P^2(\Gamma_1)} w(e)^2  \right] +
\left[ (-1)^{sg(\Gamma_2)}  \prod_{e \in P^2(\Gamma_2)} w(e)^2  \right] \end{displaymath}

If $4\not|n$, then 
\begin{displaymath} 
\hbox{det}A(C_n)=\left[- 2  \prod_{e \in E(C_n)} w(e) \right] + 
\left[ - \prod_{e \in E_1(C_n)} w(e)^2  \right] +
\left[ - \prod_{e \in E_2(C_n)} w(e)^2  \right] 
\end{displaymath} 
$$= - \left[ \prod_{e \in E_1(C_n)} w(e) \ + \prod_{e\in E_2(C_n)}  w(e) \right]^2 $$

If $4|n$, then 

\begin{displaymath} 
\hbox{det}A(C_n)=\left[- 2 \prod_{e \in E(C_n)} w(e) \right]+ 
\left[ \prod_{e \in E_1(C_n)} w(e)^2 \right] +
 \left[\prod_{e \in E_2(C_n)} w(e)^2  \right] 
\end{displaymath} 
$$= \left[\prod_{e \in E_1(C_n)} w(e) \ - \prod_{e\in E_2(C_n)}  w(e)\right]^2$$

\end{proof}

\section{Graphs $G_{v+c \cdot u}$}

Assume, that $v,u$ are distinct, non-adjacent vertices of a graph $G$, $c\neq 0$ is a real number. We define the graph $G_{v+ c \cdot u}$
in the following way \\

$V(G_{v+c \cdot u}):=V(G)$,\\
\begin{displaymath}	
w^{G_{v+c \cdot u}}(ab)=\left\{ \begin{array}{ll}
w^G(vb)+c \cdot w^G(bu)  & \textrm{if $a=v$} \\
w^G(ab) & \textrm{if  $a \neq v \neq b$}
\end{array} \right.
\hspace{3cm}
\end{displaymath} 

$e \in E(G_{v+c \cdot u})\ \textrm{iff}\ w^{G_{v+c \cdot u}}(e) \neq 0$.\\

If $G$ is a simple graph and $N_G(u) \subset N_G(v)$, then $G_{v-u}$ is also a simple graph and the operation of subtracting vertex $u$ from the vertex $v$ corresponds to the operation of removing certain edges from a graph presented in \cite{R}.
 
Notice, that if we multiply the column of the matrix $A(G)$ corresponding to the vertex $u$ by $c$ and add it to the column corresponding to the vertex $v$ and if we do the same with appropriate rows, we obtain the matrix $A(G_{v+c \cdot u})$.
Hence

\begin{tw}
	 	If $x_1$, $x_2$ are distinct and non-adjacent vertices of a graph $G$, $c \neq 0$ then 
	 	$\hbox{det} A(G)=\hbox{det} A(G_{x_1+c \cdot x_2}).$
\end{tw}

\section{Hexagonal grid graphs}

Let us consider a disjoint union $G$ of $m+1$ simple paths, such that \\

$P^i_{2n+2}=v^i_1v^i_2 \dots v^i_{2n}v^i_{2n+1}v^i_{2n+2}$ \hspace{.5cm} for $i \in \{1,2, \dots , m \}$, and \\ 

$P^{m+1}_{2n}=v^{m+1}_2v^{m+1}_3 \dots v^{m+1}_{2n+1}$.\\

We construct the graph $H_{n,m}$ from $G$ in the following way:\\
$V(H_{n,m})=V(G)$, and\\ 
$e \in E(H_{n,m})$ iff
\begin{displaymath}	
\begin{array}{lll}
e \in E(G) & \textrm{or} & e=v^{i+1}_2v^i_1, \hspace{.2cm} \textrm{ for $i \in \{1,2, \dots ,m \}$}  \\ 
\textrm{ or } & & \\
e=v^{i+1}_{2k+1}v^i_{2k+2}, & \textrm{for} & k\in \{1,2, \dots, n\}, i \in \{1,2, \dots ,m \}.  \\ 
\end{array}
\end{displaymath}

We call the $H_{n,m}$ graph a \emph{hexagonal grid}.
\begin{figure} 
\includegraphics[scale=0.30]{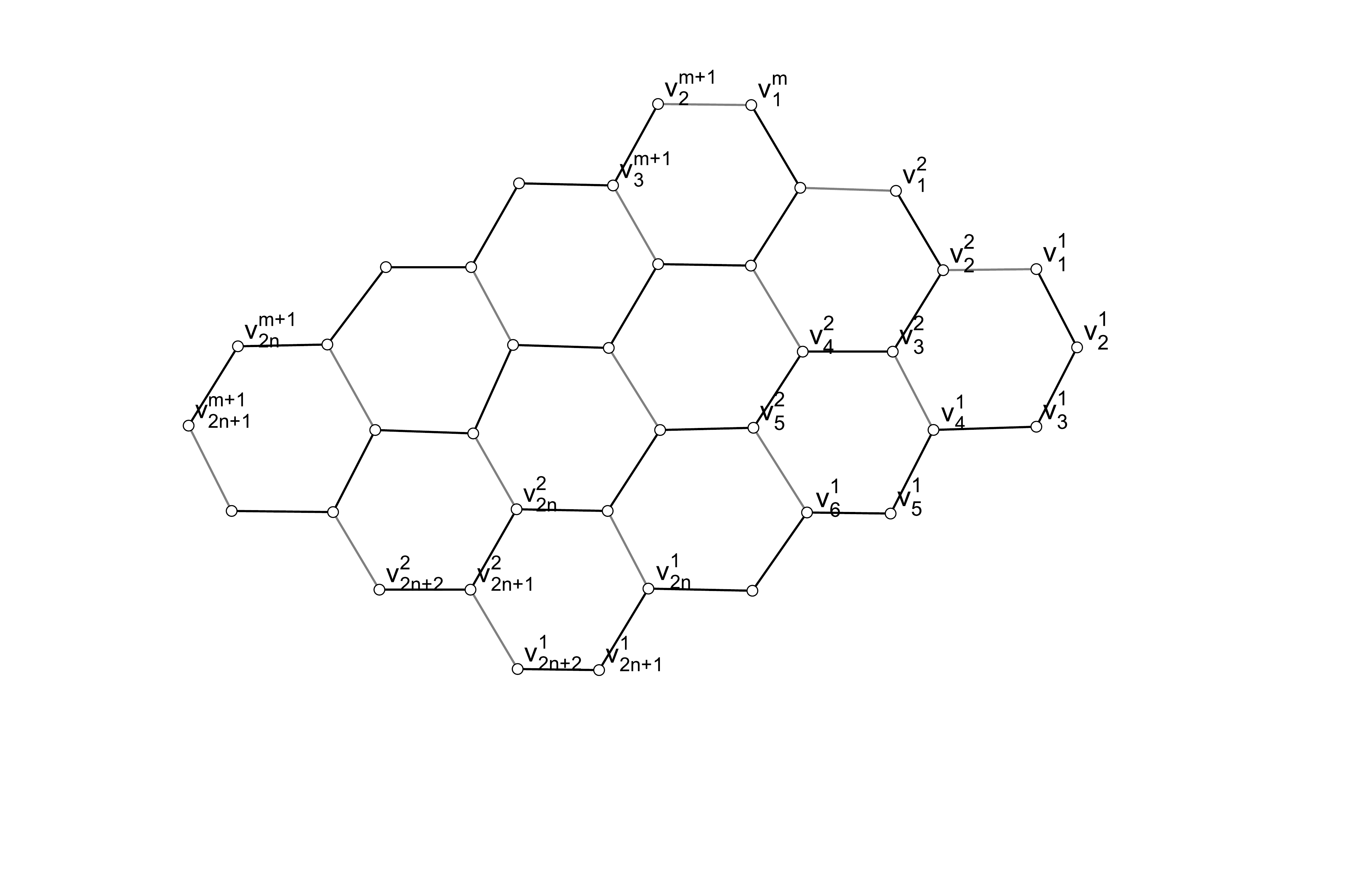}
\caption{Hexagonal grid $H_{n,m}$}
\label{he}
\end{figure}
In the figure \ref{he}, which presents the graph $H_{n,m}$, the parameters $m,n$ denote the number of cycles $C_6$ at the respective side of $H_{n,m}$.
The graph $H_{n,m}$ has a structure of a general hexagon-shaped benzenoid system $O(m,1,n)$.

Hexagonal chains $H_{n,1}$ are non-singular graphs, and \mbox{$\hbox{det} A(H_{n,1})=-(1+n)^2$} \cite{Y},\cite{G}.
We will prove a formula for calculating the determinant of any hexagonal grid $H_{n,m}$.

\begin{tw}

Let $H_{n,m}$ be a hexagonal grid $H_{n,m}$ with a weight function, such that 

($\ast$) $w(v^1_{2i}v^1_{2i+1})=\frac{x+i}{i}$ for $x \in \mathbb{N}$ and every $i \in \{1, \dots, n \}$. Weights of the other edges are equal to 1. (see fig. \ref{d1})\\

Then 
$$\hbox{det} A(H_{n,m})=(-1)^{nm+n+m}\left[ {x+n+m \choose n}\right]^2$$

\end{tw}

\begin{proof}

\begin{figure} 
\includegraphics[scale=0.30]{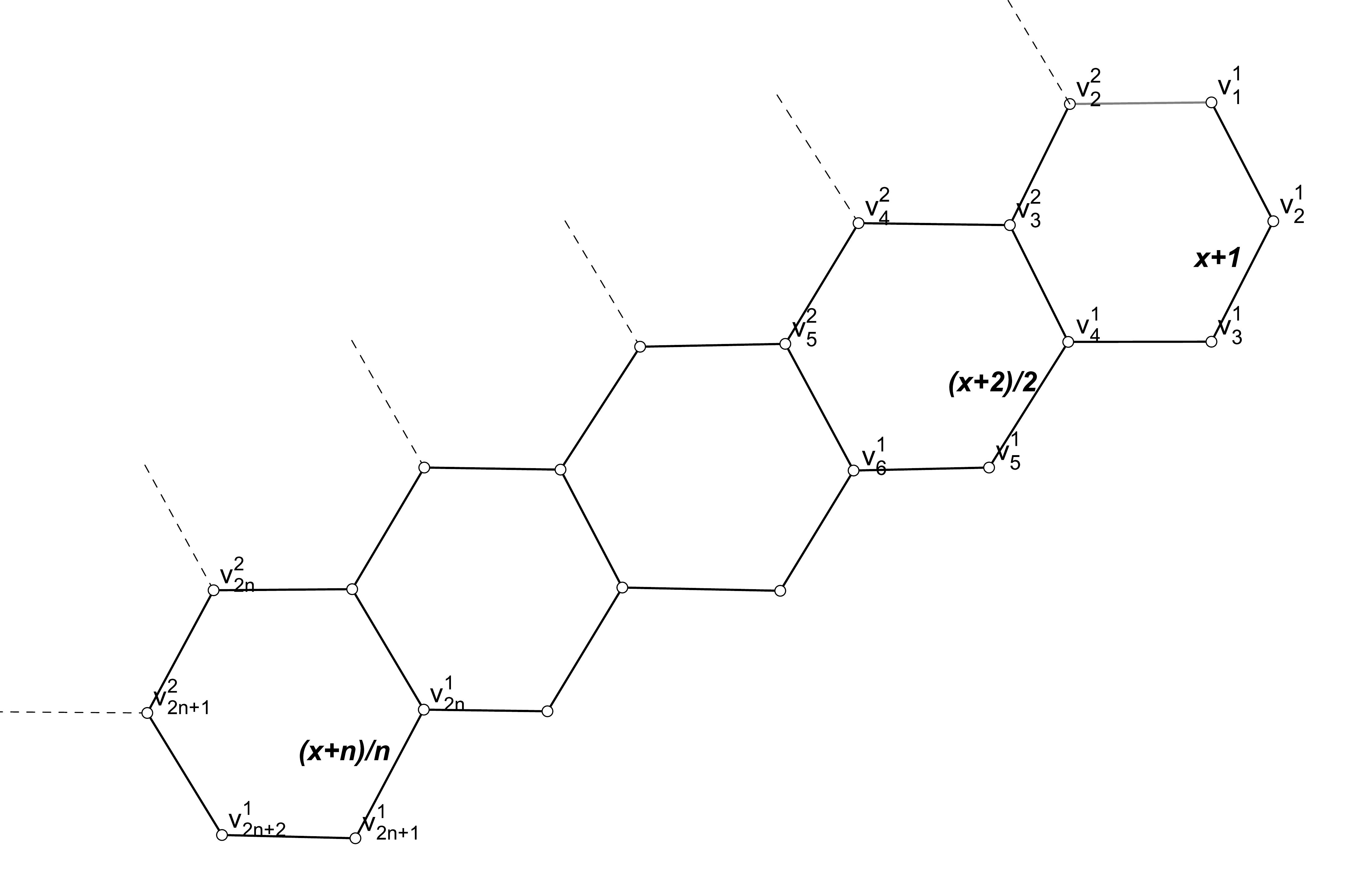}
\caption{$H_{n,m}$}
\label{d1}
\end{figure}

Let us consider $(H_{n,m})_{v^2_2-v^1_2}$. 
In this graph $v^1_1$ is a pending vertex. According to corollary \ref{ped}, the edge adjacent to its neighbour can be removed.
If  $$G_0=((H_{n,m})_{v^2_2-v^1_2})\setminus \{v^1_2v^1_3 \}$$ then 
$w^{G_0}(v^2_2v^1_3)=-(x+1)$.
In graph $$G_1=(G_0)_{v^2_3-v^1_3}$$ we have $w^{G_1}(v^2_2v^2_3)=2+x$.  
Let us consider the graph $(G_1)_{v^1_5-\frac{x+2}{2}\cdot v^1_3}$. 
For $v^1_4$ is a pedant vertex in this graph, we can apply the corollary \ref{ped}.
Let $$G_2=(G_1)_{v^1_5-\frac{x+2}{2}\cdot v^1_3} \setminus \{v^1_3v^2_2 \}$$ then $w^{G_2}(v^2_2v^1_5)=\frac{(x+1)(x+2)}{2}$.
Now we construct the graph $$G_3=(G_2)_{v^1_5-\frac{x+1}{2} \cdot v^2_3}$$ 
We get $w^{G_3}(v^2_4v^1_5)=-\frac{x+1}{2(2+x)}$. 
\begin{figure} 
\includegraphics[scale=0.30]{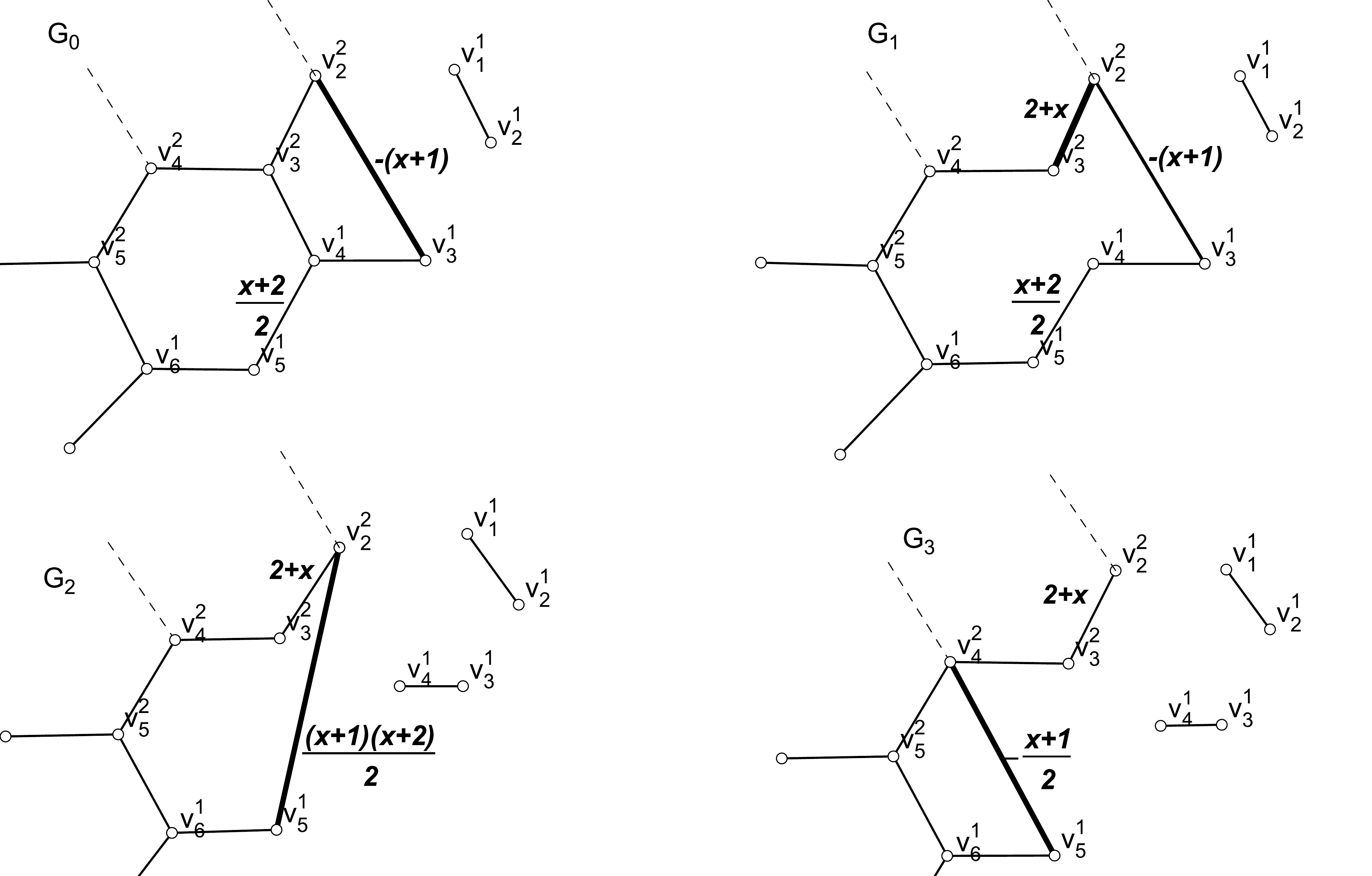}
\caption{Reduction of $H_{n,m}$}
\label{d2}
\end{figure}

Let us assume, that we constructed graphs $G_{3k-2}$, $G_{3k-1}$ and $G_{3k}$, for some $k\in \{1,\dots , n-2  \}$, and  
($\ast \ast$) $G_{3k}$ is a graph, which can be obtained from $H_{n,m}$ by 
\\
- removing the edges $v^1_{2i}v^1_{2i+1}$ for $1 \leq i \leq k+1$; $v^1_2iv^2_{2i-1}$ for $2 \leq i \leq k+1$,\\ and $v^1_1v^2_2$, \\
- setting the weights: 
$w(v^2_{2i}v^2_{2i+1})=\frac{x+1+i}{i}$ for $1 \leq i \leq k-1$,\\
- adding a new edge with the weight $w(v^2_{2k+2}v^1_{2k+3})=-\frac{x+1}{k}$.

Let $$G_{3k+1}=(G_{3k})_{v^2_{2k+3}-v^1_{2k+3}}$$ then $w^{G_{3k+1}}(v^2_{2k+2}v^2_{2k+3})=\frac{k+1+x}{k}$.  
Now we construct the graph $$(G_{3k+1})_{v^1_{2k+5}-\frac{x+k+2}{k+2}\cdot v^1_{2k+3}}$$ 
For $v^1_{2k+4}$ is a pedant vertex in this graph, we can apply the corollary \ref{ped}.
Let $$G_{3k+2}=(G_{3k+1})_{v^1_{2k+5}-\frac{x+k+2}{k+2}\cdot v^1_{2k+3}} \setminus \{v^1_{2k+3}v^2_{2k+2} \}$$ $w^{G_2}(v^2_{2k+2}v^1_{2k+5})=\frac{(x+1)(x+k+1)}{k(k+1)}$.
Now we construct the graph $$G_{3k+3}=(G_{3k+2})_{v^1_{2k+5}-\frac{x+1}{k+1)} \cdot v^2_{2k+3}}$$  
We get $w^{G_{3k+3}}(v^2_{2k+4}v^1_{2k+5})=-\frac{x+1}{k+1}$.

From the mathematical induction we conclude, that for every $k<n$, $G_{3k}$ is a graph satisfying conditions ($\ast \ast$).
Notice, that $v^1_{2n+2}$ is a pending vertex in the graph $G_{3(n-1)+1}$, and if $n>1$, then $F=G_{3(n-1)+1} \setminus \{v^1_{2n+1}v^2_{2n}\}$
is a disjoint union of $n+1$ simple $P_2$ paths and a weighted hexagonal grid $H_{n,m-1}$, such that  $w(v^2_{2i}v^2_{2i+1})=\frac{x'+i}{i}$ for every $i \in \{1, \dots, n \}$ and $x'=x+1$.

\begin{figure} 
\includegraphics[scale=0.30]{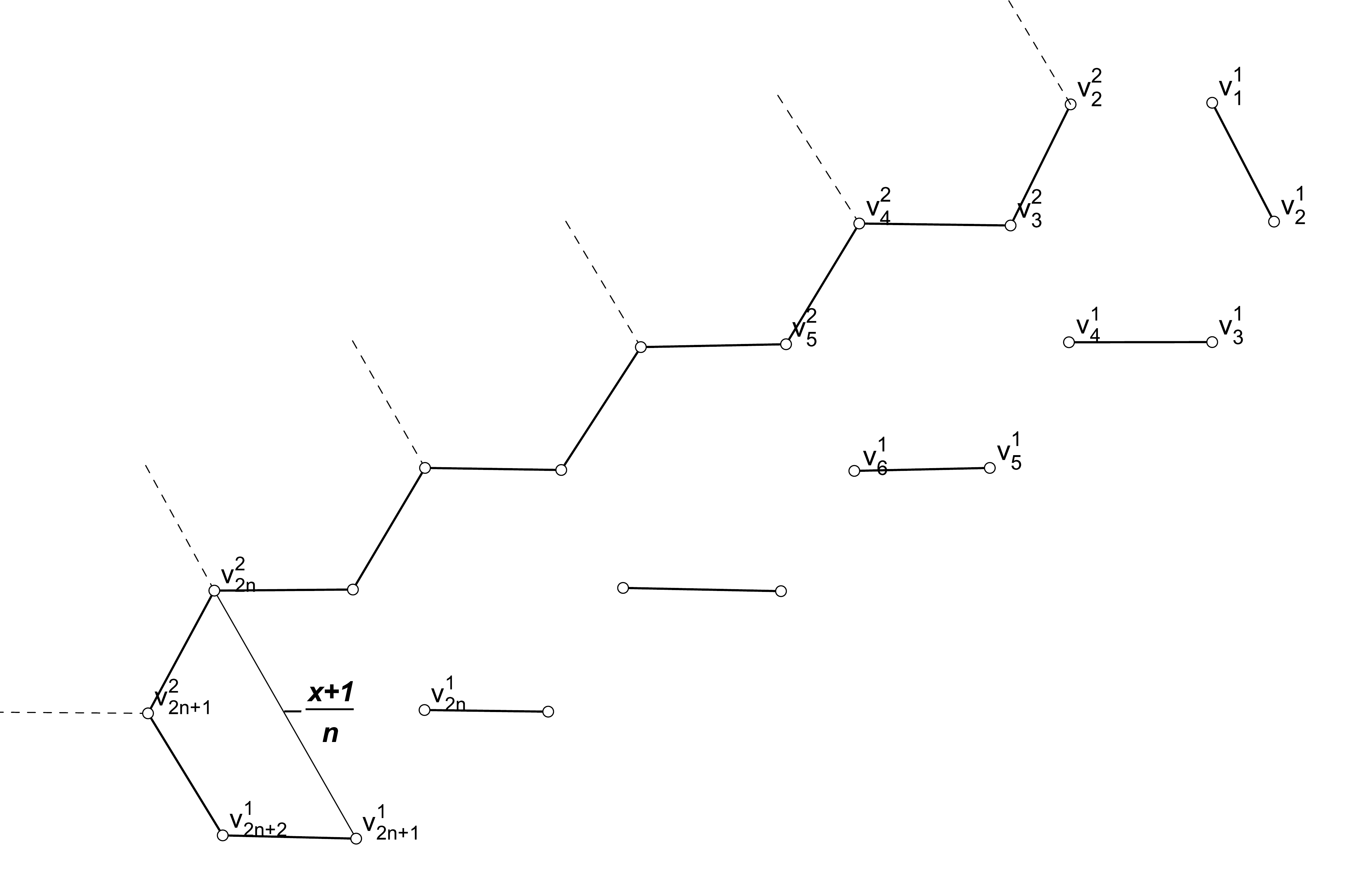}
\caption{$G_{3(n-1)}$}
\label{d7}
\end{figure}

If $m=1$ then $F$
is a disjoint union of $n+1$ simple $P_2$ paths and a weighted path $P_{2n}=v^2_2v^2_3 \dots v^2_{2n+1}$, such that $w(v^2_{2i}v^2_{2i+1})=\frac{x'+i}{i}$  for every $i \in \{1, \dots, n \}$ and  $x'=x+1$. In this case the only sesquivalent spanning subgraph is a disjoint union of $2n+1$ paths $P_2$, and hence $$\hbox{det} A(H_{n,1})= - \left[\frac{(n+1+x)!}{(1+x)!n!}\right]^2$$

Assume that if a hexagonal grid $H_{n,m-1}$ satisfies the conditions ($\ast$), then
$$\hbox{det} A(H_{n,m-1})= (-1)^{(m-1)n+n+(m-1)} \left[\frac{(n+m-1+x)!}{n!(x+m-1)!}\right]^2$$

We re-label each vertex of F from $v^i_k$ to $v^{i-1}_k$. The component of F which is a hexagonal grid satisfies the conditions ($\ast$).
Hence $$\hbox{det} A(F)=(-1)^{n+1} \cdot (-1)^{mn+m-1} \left[\frac{(n+m-1+x')!}{n!(x'+m-1)!}\right]^2$$
We obtain $$\hbox{det} A(F)=(-1)^{mn+n+m} \left[\frac{(n+m+x)!}{n!(x+m)!}\right]^2$$
For the reductions do not change the determinant of the graph, we conclude, that 
$$\hbox{det} A(H_{n,m})=(-1)^{nm+n+m}\left[ {x+n+m \choose n}\right]^2$$ 

\end{proof}

\begin{tw}
If $k,n \geq 1$, then $$\hbox{det} A(H_{n,k})= (-1)^{kn+n+k} \left[{n+k \choose n}\right]^2$$
\end{tw}

\end{document}